\documentclass[12pt,a4paper]{article}

\usepackage{amsmath,amssymb,amsthm}
\usepackage[shortlabels]{enumitem}
\setlist[enumerate]{label=(\roman*)}
\usepackage{graphicx}
\usepackage{subcaption}
\usepackage[colorlinks=false]{hyperref}
\hypersetup{
    colorlinks,
    linkcolor={red!50!black},
    citecolor={blue!50!black},
    urlcolor={blue!80!black}
}
\usepackage[capitalise]{cleveref}
\usepackage{xcolor}
\usepackage{natbib}
\setcitestyle{semicolon}

\usepackage[colorinlistoftodos]{todonotes}
\usepackage[framemethod=tikz]{mdframed}

\newtheorem{theorem}{Theorem}[section]
\newtheorem{lemma}[theorem]{Lemma}
\newtheorem{prop}[theorem]{Proposition}

\theoremstyle{definition}

\newtheorem{example}[theorem]{Example}
\newtheorem*{remark}{Remark}

\newtheorem{conj}{Conjecture}
\numberwithin{equation}{section}

\newcommand \dm[1]  { \,\mathrm d{#1} }

\renewcommand{\epsilon}{\varepsilon}
\renewcommand{\phi}{\varphi}

\renewcommand{\leq}{\leqslant}
\renewcommand{\geq}{\geqslant}

\newcommand{\norm}[1]{\left\lVert#1\right\rVert}

\newcommand{\norml}[3]{\norm{#1}_{L^{#2}({#3})}}

\newcommand{\abs}[1]{\left\vert#1\right\vert}

\renewcommand{\hat}[1]{\widehat{#1}}

\renewcommand{\tilde}{\widetilde}
\newcommand{\e}{\mathrm{e}}

\newcommand{\tpi}{2 \pi \mathrm{i}}

\newcommand{\R}{\mathbb{R}}

\newcommand{\C}{\mathbb{C}}
\newcommand{\N}{\mathbb{N}}

\newcommand{\Z}{\mathbb{Z}}

\newcommand{\D}{\mathbb{D}}

\renewcommand{\th}{\textsuperscript{th} }
\newcommand{\qtext}[1]{\quad\text{#1}\quad}
\newcommand{\qand}{\qtext{and}}
\newcommand{\qfa}{\qtext{for all}}
\newcommand{\ie}{i.e.\ }

\allowdisplaybreaks

\title{A partial converse to the Riemann--Lebesgue lemma for Bessel--Fourier series of order zero}
\author{Ryan L.~Acosta Babb\footnote{University of Warwick, \href{ryan.l.acosta-babb@warwick.ac.uk}{ryan.l.acosta-babb@warwick.ac.uk}}}
\date{}

\begin{document}
\maketitle

\begin{abstract}
    It is known that the Bessel--Fourier coefficients $f_m$ of a function $f$ such that $\sqrt{x}f(x)$
    is integrable over $[0,1]$ satisfy $f_m/\sqrt{m}\to 0$.
    We show a partial converse, namely that for $0\leq \alpha<1/2$
    and any non-negative $a_m\to 0$, there is a function $f$
    such that $x^{\alpha+1}f(x)$ is integrable and its Bessel--Fourier coefficients $f_m$ satisfy $m^{-\alpha}f_m\geq a_m$
    and $m^{-\alpha}f_m\to 0$.
    We conjecture that the same should be true when $\alpha=\frac{1}{2}$,
    and discuss some consequences of this conjecture.
\end{abstract}

\section{Introduction}
The Riemann--Lebesgue Lemma is a classic result in Fourier analysis,
which tells us that the Fourier coefficients of an integrable function, $f$,
decay to 0 at high frequencies, \ie $\hat{f}(n)\to 0$ as $\abs{n}\to\infty$.

By increasing the regularity of $f$, say $f\in C^k$, we can obtain faster rates
of decay of the coefficients, \emph{viz.} $\abs{n}^{k}\hat{f}(n)\to 0$.
But if we merely assume that $f\in L^1$, then the Riemann--Lebesgue lemma
is the best possible result we can expect.
Indeed, we can find integrable functions whose Fourier coefficients
decay as slowly as we like, as the following theorem shows \citep[see, e.g.,][Theorem 3.3.4]{GrafakosCFA}.
\begin{theorem}\label{thm:ConRL}
    Let $a_n\geq 0$ be a sequence such that $a_n\to 0$ as $\abs{n}\to\infty$.
    Then, there is a function $f\in L^1([0,1])$
    whose Fourier coefficients satisfy $\hat{f}(n)\geq a_n$ for all $n\in\Z$.
\end{theorem}

This may seem `obvious': if we have the coefficients $a_n$, can we not just set \[
    f(x) = \sum_{n=-\infty}^{\infty}a_n\e^{\tpi nx}?
\] The answer is no, since there is no guarantee that this series will converge to an $L^1$ function.
Recall that the \emph{Fej\'er kernel} is given by
    \begin{equation}\label{eqn:fejer}
        F_{N}(x) = \sum_{n=-N}^{N}\left(1-\frac{\abs{n}}{N+1}\right)\e^{\tpi nx}
            \equiv \frac{1}{N+1}\frac{\sin(\pi(N+1)x)^2}{\sin(\pi x)^2},
    \end{equation}
and consider the function \[
        f(x) := \sum_{m=0}^{\infty}\frac{1}{2^m}F_{2^{2^{2^m}}}(x).
\] Since $F_N(x)\geq 0$ and $\norml{F_N}{1}{[0,1]}=1$ for all $N$,
this function is clearly well-defined and integrable.
However, as Exercises 4.2.3 and 4.2.4 in \cite{GrafakosCFA} show,
the partial sums of the Fourier series of $f$ diverge (in the $L^1$ norm)
like $N\log(N)$.

Note that the Fourier coefficients of $f$ are the nasty-looking \[
\hat{f}(n) = \sum_{m=\lceil\log_2\log_2\log_2(\abs{n})\rceil}^{\infty} \left(
    \frac{1}{2^m} - \frac{\abs{n}}{2^{2^{2^n}}+1}\right),
\] which nevertheless tend to $0$ as $\abs{n}\to\infty$ by the Riemann--Lebesgue Lemma.

In some situations, Fourier series may not be the ideal choice to expand our functions.
For instance, when working on the disc
$$\D:=\{(r\cos(2\pi\theta),r\sin(2\pi\theta)) : r,\theta\in[0,1]\},$$
it is convenient to expand $f\colon \D\to\C$ using eigenfunctions of the Laplacian on $\D$.
This leads to the (2-dimensional) \emph{Bessel--Fourier series} of $f$,
\begin{equation}\label{eqn:BFseries}
    \sum_{n=-\infty}^{\infty}\sum_{m=1}^{\infty}f_{n,m}J_{\abs{n}}(j_{\abs{n},m}r)\e^{\tpi n\theta},
\end{equation} where $J_{n}$ are Bessel functions of the first kind (of order $n\in\N$),
$j_{n,m}>0$ ($m\in\N$) are the zeros of $J_n$, and $f_{n,m}$ are the \emph{Bessel--Fourier coefficients} \[
    f_{n,m} := \frac{2}{J_{\abs{n}+1}(j_{\abs{n},m})^2}\int_0^1\int_0^1
        f(r,\theta)J_{\abs{n}}(j_{\abs{n},m}r)\e^{\tpi n\theta}r\dm{r}\dm{\theta}.
\] Note the $r\dm{r}$ in the measure coming from the use of polar coordinates on $\D$.

If $f$ is radial, \ie $f(r,\theta) \equiv f(r)$ independently of $\theta$,
then the only non-zero (angular) Fourier coefficient occurs for $n=0$.
In this manner, we obtain a one-dimensional Bessel--Fourier series for $f$ of order zero:
\begin{equation}
    \sum_{m=1}^{\infty}f_mJ_0(j_m x),
\end{equation} where we have set $j_m := j_{0,m}$ for notational convenience,
and the Bessel--Fourier coefficients now take the from \[
    f_m := \frac{2}{J_{1}(j_{m})^2}\int_0^1 f(x)J_{0}(j_{m}x)x\dm{x}.
\]

The corresponding `Riemann--Lebesgue lemma' for the $f_m$ is as follows \citep{Young1920}.
\begin{theorem}\label{thm:RLforBF}
    Fix $0\leq s \leq \frac{1}{2}$.
    If $x^{1/2+s}f(x)\in L^1([0,1])$ and $f_m$ are the Bessel--Fourier coefficients of $f$,
    then \[
        \lim_{m\to\infty} \frac{f_m}{j_m^{1/2+s}} = 0.
    \]
\end{theorem}
This can be proved by recalling that $J_0(x)$ behaves asymptotically like $\frac{1}{\sqrt{x}}\cos(x)$;
the $\sqrt{x}$ is controlled by the condition $x^{1/2+s}f(x)\in L^1$,
and one can apply the classical Riemann--Lebesgue Lemma for cosines to the function $x^sf(x)$.


Of particular interest are the following special cases:
\begin{align*}
    \sqrt{x}f(x) \in L^1 &\implies \lim_{m\to\infty} \frac{f_m}{\sqrt{j_m}} = 0\\
    xf(x) \in L^1 &\implies \lim_{m\to\infty} \frac{f_m}{j_m} = 0.
\end{align*}

In light of \cref{thm:RLforBF}, we might suspect that we can find $L^1$ functions
whose coefficients can be made to blow up
at any given rate below $\sqrt{j_m}$ (whenever $\sqrt{x}f(x)\in L^1$) or even $j_m$
(if $xf(x)\in L^1$).
The following example supports our suspicions.
\begin{example}\label{ex:falpha}
    Consider the functions $f^{(\alpha)}(x)=x^{-\alpha}$ for $\alpha>0$, so that
    $xf^{(\alpha)}(x)\in L^1$ if, and only if, $\alpha<2$.
    Taking $\alpha = 2-\epsilon$, the functions $f^{(\alpha)}$ are in $L^1(x\dm{x})$
    for all $\epsilon>0$.
    Their Bessel--Fourier coefficients can be estimated as follows
    (where `$\approx$' is a shorthand for `neglecting constants'):
    \begin{align*}
        f^{(\alpha)}_m &= \frac{2}{J_1(j_m)^2}\int_0^1 x^{-2+\epsilon}J_0(j_m x)x\dm{x}\\
                       &\approx j_m\int_0^{j_m}j_m^{1-\epsilon}\frac{J_0(y)}{y^{1-\epsilon}}\frac{\dm{y}}{j_m}\\
                       &\approx j_m^{1-\epsilon}\left(\int_0^{j_1}\frac{J_0(y)}{y^{1-\epsilon}}\dm{y}
                            +\int_{j_1}^{j_m}\frac{\cos(y-\pi/4)}{y^{3/2-\epsilon}}\dm{y}
                            +\int_{j_1}^{j_m}O(y^{-5/2+\epsilon})\dm{y}\right)
    \end{align*} The first integral is clearly finite;
    the third is clearly bounded uniformly in $j_m$.
    Finally, one can exploit the oscillations in the cosine to show that the second integral is also bounded
    (see \cref{sec:KMbounds} or \citealp{Zygmund2003}).
    Therefore, $f^{(\alpha)}_m = O(j_m^{1-\epsilon})$.

    In particular, taking $\epsilon=\frac{1}{2}$, we have $f^{(1/2)}(x) = x^{-3/2}$,
    so that $xf^{(1/2)}(x)\in L^1$ but $\sqrt{x}f^{(1/2)}(x)\not\in L^1$.
    In this case, the coefficients decay as $\sqrt{j_m}$.
    This example shows that $\sqrt{j_m}$ is a `limiting rate of growth' for functions in $L^1(\sqrt{x}\dm{x})$.
\end{example}

The decay of Bessel--Fourier coefficients has been exploited recently by \cite{SaadiDaher2022,SaadiDaher2023}
to study when a function belongs to a Lipschitz class.

In this paper we prove the following analogue of \cref{thm:ConRL} for
Bessel--Fourier series,
which is a partial converse to \cref{thm:RLforBF}.

\begin{theorem}\label{thm:mainthm}
    Let $0\leq \alpha <\frac{1}{2}$ and
    $a_m\geq 0$ be a sequence with $a_m\to 0$.
    Then, there is a function $f$, with $x^{\alpha+1}f(x)\in L^1$,
    whose Bessel--Fourier coefficients of order zero, $f_m$, satisfy \[
        0 \leq a_m \leq \frac{f_m}{j_m^\alpha} \to 0\qtext{as} m \to \infty.
    \] Furthermore, $f$ can be chosen so that $f_{2m}=0$ for all $m\in\N$,
    or so that $f_{2m+1} = 0$ for all $m\in\N$.
\end{theorem}

We also show, in \cref{sec:conc}, how to construct an $L^1$ function on the disc whose Bessel--Fourier
series diverges in norm for any possible truncation of the partial sums,
assuming that \cref{thm:mainthm} can be improved to coefficients with $a_m = o(\sqrt{j_m})$.
Here, the condition $f_{2m+1}=0$ will be crucial.

Note that \cref{thm:mainthm} does not prove that \cref{thm:RLforBF} is `sharp'.
Indeed,
\begin{itemize}
    \item in \cref{thm:RLforBF} we need only assume that $x^{\alpha+1/2}f(x) \in L^1$
    to get a decay of $f_m = o(j_m^{\alpha+1/2})$, for all $0\leq \alpha \leq \frac{1}{2}$;
    \item in \cref{thm:mainthm} we construct an $f$ with $x^{\alpha+1}f(x)\in L^1$
    with a decay of $f_m = o(j_m^\alpha)$ for $0\leq \alpha < \frac{1}{2}$.
\end{itemize}
Our result requires an extra $\sqrt{x}$, loses a factor of $\sqrt{j_m}$ in the decay,
and leaves the end-point $\alpha=\frac{1}{2}$ open.

However, assuming that a kernel analogous to $F_N$ is positive, we can show a stronger result,
namely, that for any $0\leq \alpha < \frac{1}{2}$, there is an 
$f$ with $\sqrt{x}f(x) \in L^1$ whose coefficients decay as $f_m = o(j_m^\alpha)$,
which is in line with \cref{thm:RLforBF} for $s=0$.
This positivity conjecture is supported by numerical calculations.

The paper is organised as follows:
in \cref{sec:bg} we record some necessary facts about
Bessel functions and their zeros.
In \cref{sec:wdthm}, we give a rigorous proof of \cref{thm:mainthm},
without assuming any additional results;
while in \cref{sec:sdthm} we formulate and prove a stronger
version of \cref{thm:mainthm}, assuming a plausible conjecture.
\cref{sec:KMbounds} contains a proof of a technical, but crucial
lemma used in the proof of \cref{thm:mainthm}.
Finally, in \cref{sec:conc} we discuss how these methods
could be applied to construct an $L^1$ function on the disc
whose Bessel--Fourier series diverges regardless of how one chooses
to truncate the partial sums.

\paragraph{Notation.} Throughout, we will employ the following asymptotic notation:
\begin{align*}
    f(x) &\lesssim g(x) &\iff &f(x) \leq Cg(x) \text{ for some absolute constant } C>0\\
    f(x) &\approx g(x) &\iff &f(x) \lesssim g(x) \qand g(x) \lesssim f(x)\\
    f(x) &= O(g(x)) &\iff &\abs{f(x)} \lesssim g(x)\\
    a_m &= o(b_m) &\iff &\frac{a_m}{b_m} \to 0 \qtext{as} m \to \infty.
\end{align*}

We adopt the convention whereby $\N$ is the set of \emph{positive} integers,
with $\N_0 := \N\cup\{0\}.$
\section{Bessel functions and their asymptotics}\label{sec:bg}
In order to prove \cref{thm:mainthm}, we will need the following facts about
the Bessel function $J_0$.

\begin{lemma}\label{lem:Jasym}
    For any $r>0$, we have
    \begin{equation}\label{eqn:J0asym}
        J_0(r) = \sqrt{\frac{2}{\pi}}\frac{\cos(x-\pi/4)}{\sqrt{x}} + R(r),
    \end{equation}
    where
    \begin{equation}\label{eqn:J0err}
        \abs{R(r)} \lesssim \int_{0}^{2}\e^{-rt}\sqrt{t}\dm{t}
            + \int_{2}^{\infty}\frac{\e^{-rt}}{t}\dm{t}.
    \end{equation}
    Furthermore, \begin{equation}\label{eqn:J0errorlarger}
        \abs{R(r)} \lesssim r^{-3/2} \qtext{for} r \geq 1.
    \end{equation}
\end{lemma}
These are well-known results; see for instance Appendinx B in \cite{GrafakosCFA}.
The specific form of the bound (\ref{eqn:J0err})
is derived in detail in \cite{AcostaBabb2024}.

We will also use the following information about the zeros of $J_0$,
and how they interact with $J_1$.
\begin{lemma}\label{lem:zeros}
    Let $j_m$ denote the $m\th$ positive zero of $J_0$. Then
    \begin{equation}\label{eqn:jmm}
        j_m > m \qfa m \geq 1.
    \end{equation}
    Furthermore,
    \begin{equation}\label{eqn:jmJ1jmasym}
        J_1(j_m) = O\left(\frac{1}{\sqrt{j_m}}\right).    
    \end{equation}
\end{lemma}
\begin{proof}
    The estimate in \cref{eqn:jmJ1jmasym} follows from the asymptotics
    for $J_1$; see \cite{Tolstov1976}.
    On the other hand, (\ref{eqn:jmm}) is easily obtained from the following inequality \citep[Theorem 3]{Hethcote1970} \[
        j_{m} \geq \pi m  - \frac{\pi}{4} \qfa m \geq 1. \qedhere
    \]
\end{proof}
\begin{remark}
    Numerical calculations suggests that the sequence $J_1(j_m)$ oscillates
    and that $\abs{J_1(j_m)}$ is decreasing.
    We do not at present have a proof of these facts,
    but will refer to them again later when discussing possible generalisations and applications
    of our results.
\end{remark}

Finally, we record two useful formulas, which will be needed in several calculations.
The first is
\begin{equation}\label{eqn:J0prime}
    J'_0(x) = - J_1(x) \qfa x \in \R,
\end{equation}
and can be found in Appendix B to \cite{GrafakosCFA}.
The second is
\begin{equation}\label{eqn:intJ0rdr}
    \int_0^{j_m}J_0(y)y\dm{y} = j_mJ_1(j_m) \qfa m \in \N.
\end{equation}
To see this, recall that $J_0$ satisfies Bessel's differential equation: \[
    x^2J_0''(x) + xJ_0'(x) + x^2J_0(x) = 0,
\] and that $J_0'(x) \equiv -J_1(x)$.
Integrating the differential equation therefore leads us to \cref{eqn:intJ0rdr}.

\section{`Weak' decay theorems}\label{sec:wdthm}
In order to construct a function whose Bessel--Fourier coefficients have the desired
decay properties, we will follow the argument of
Theorem 3.3.4 in \cite{GrafakosCFA}.

Recall that a sequence $(c_m)_m$ is \emph{convex} if \[
    c_m \geq 0 \qand c_{m} + c_{m+2} \geq 2 c_{m+1} \qfa m \in \N.
\] The following two lemmas are proved in Chapter 3 of \cite{GrafakosCFA}.
\begin{lemma}[Lemma 3.3.2 in \citealp{GrafakosCFA}]\label{lem:conv}
    Given a sequence $a_m\geq 0$ tending to zero as $m\to\infty$,
    there is a convex sequence $c_m\geq a_m$ that also tends to zero.
\end{lemma}
\begin{lemma}[Lemma 3.3.3 in \citealp{GrafakosCFA}]\label{lem:telescope}
    Given a convex decreasing sequence $c_m\to 0$ and a fixed integer $k\geq 0$, \[
        \sum_{l=0}^{\infty}(l+1)(c_{l+k}+c_{l+k+2}-2c_{l+k+1}) = c_k.
    \]
\end{lemma}

Let us write $$S^{\alpha}_m(x) = \sum_{k=1}^{m}j_m^{\alpha}J_0(j_k x).$$
We will need the following `modified Fej\'er kernel' for the Bessel--Fourier series:
\[
    K^{\alpha}_M(x) := \frac{1}{M+1}\sum_{m=1}^{M}S^{\alpha}_m(x)
            = \sum_{m=1}^{M}\left(1-\frac{m}{M+1}\right)j_m^{\alpha}J_0(j_m x).
\] We will also set $K^{\alpha}_0(x) \equiv 1$ when $M=0$.

As we saw in \cref{eqn:fejer}, the classical Fej\'er kernel for Fourier series
has a nice closed form, which shows it to be non-negative for all $x$ and $N$;
this makes computing $L^1$ norms much easier.

To prove \cref{thm:mainthm}, it suffices to obtain uniform lower bounds on
$x^{\alpha+1}K^\alpha_M(x)$ for $0\leq \alpha<\frac{1}{2}$.
\begin{lemma}\label{lem:xFbounds}
    For $0\leq\alpha < \frac{1}{2}$, there is a constant $C>0$ such that \[
        x^{\alpha+1}K_{M}^{\alpha}(x) \geq -C \qtext{uniformly in} M\in\N \qand x\in[0,1].
    \]
\end{lemma}
We will prove \cref{lem:xFbounds} in \cref{sec:KMbounds}.
Here, we will re-state and prove \cref{thm:mainthm}.
\begin{theorem}[$=$ \cref{thm:mainthm}]
    Let $0\leq \alpha <\frac{1}{2}$ and
    $a_m\geq 0$ be a sequence such that $a_m\to 0$.
    Then, there is a function $f$, with $x^{\alpha+1}f(x)\in L^1$,
    whose Bessel--Fourier coefficients of order zero, $f_m$, satisfy \[
        0 \leq a_m \leq \frac{f_m}{j_m^\alpha} \to 0
            \qtext{as} m \to \infty.
    \] Furthermore, $f$ can be chosen so that $f_{2m}=0$ for all $m\in\N$,
    or so that $f_{2m+1} = 0$ for all $m\in\N$.
\end{theorem}
\begin{proof}
    By \cref{lem:conv} there is a decreasing convex sequence $c_m \geq a_m$
    such that $c_m\to 0$.
    (Note that we may need to take $c_0 = 2c_1$, say, in order
    to have a decreasing, convex sequence defined for $m\geq 0$.)
    We use $(c_m)_{m\geq 0}$ to construct the function
    \begin{equation}\label{eqn:fdefn}
        f(x) := \sum_{l=1}^{\infty}(l+1)(c_{l}+c_{l+2}-2c_{l+1})K_l^\alpha(x).
    \end{equation}
    We claim that $x^{\alpha+1}f(x)\in L^1$.

    First, by \cref{lem:xFbounds}, $x^{\alpha+1}K^\alpha_M(x)+C\geq 0$, so we have \[
        \norml{x^{\alpha+1}K^\alpha_M}{1}{[0,1]} \leq C + \int_0^1 (K^\alpha_M(x)x^{\alpha+1}+C)\dm{x} = 2C + \int_0^1 K^\alpha_M(x)x^{\alpha+1}\dm{x}.
    \]
    Writing $K^\alpha_M$ out as a sum, and changing variables to $y=j_mx$ in the integral, we have
    \begin{align}
        \int_0^1 K^\alpha_M(x)x^{\alpha+1}\dm{x}
            &= \sum_{m=1}^M\left(1-\frac{m}{M+1}\right)j_m^{\alpha}\int_0^1 J_0(j_m x)x^{\alpha+1}\dm{x}\nonumber\\
            &= \sum_{m=1}^M\left(1-\frac{m}{M+1}\right)\frac{j_m^{\alpha}}{j_m^{\alpha+2}}
                    \int_0^{j_m}J_0(y)y^{\alpha+1}\dm{y}. \label{eqn:intxKM}
    \end{align}

    We now split the integral into the ranges $0\leq y \leq 1$ and $1\leq y \leq j_m$,
    and use the asymptotics from \cref{lem:Jasym}:
    \begin{align}
        \frac{1}{j_m^2}\int_0^{j_m}J_0(y)y^{\alpha+1}\dm{y}
            &= \frac{1}{j_m^2}\int_0^{1}J_0(y)y^{\alpha+1}\dm{y}\label{eqn:int01alpha}\\
            &\quad+ \frac{1}{j_m^2}\int_1^{j_m}\sqrt{\frac{2}{\pi}}\cos(y-\pi/4)y^{\alpha+1/2}\dm{y}\label{eqn:int1jmcosalpha}\\
            &\quad+ \frac{1}{j_m^2}\int_1^{j_m}O\left(y^{\alpha-1/2}\right)\dm{y}\label{eqn:int1jmRalpha}.
    \end{align}
   Now, for (\ref{eqn:int01alpha}), we immediately have \[
        \frac{1}{j_m^{2}}\int_0^{1}J_0(y)y^{\alpha+1}\dm{y} \lesssim \frac{1}{j_m^{2}},
   \] and for (\ref{eqn:int1jmRalpha}), \[
        \frac{1}{j_m^{2}}\int_1^{j_m}O\left(y^{\alpha-1/2}\right)\dm{y}
            \lesssim \frac{1}{j_m^{2}}\left(
                j_m^{\alpha+1/2}-1\right) \lesssim \frac{1}{j_m^{3/2-\alpha}}.
   \] To get suitable decay in (\ref{eqn:int1jmcosalpha}), we have to integrate by parts.
   Neglecting nuisance factors of $\frac{2}{\pi}$, we have
   \begin{align*}
       \frac{1}{j_m^{2}}\int_1^{j_m}\cos(y-\pi/4)y^{\alpha+1/2}\dm{y}
        &= \frac{1}{j_m^{2}}\left[\sin(y-\pi/4)y^{\alpha+1/2}\right]_1^{j_m}\\
            &\quad-\frac{\alpha+1/2}{j_m^{2}}\int_1^{j_m}\sin(y-\pi/4)y^{\alpha-1/2}\dm{y}\\
        &\lesssim \frac{j_m^{\alpha+1/2}}{j_m^{2}} + \frac{1}{j_m^{2}}
            + \frac{1}{j_m^{2}}\\
        &\lesssim \frac{1}{j_m^{3/2-\alpha}}.
   \end{align*}
   (For the boundedness of the sine integral, see the argument in \cref{sec:KMbounds}.)
   Since $\alpha<1/2$, the terms in \cref{eqn:int01alpha,eqn:int1jmcosalpha,eqn:int1jmRalpha} are summable,
   and, therefore, so is (\ref{eqn:intxKM}).
    In other words, $\norml{x^{\alpha+1}K^\alpha_M}{1}{[0,1]} = O(1)$ uniformly in $M$,
    and so, from \cref{eqn:fdefn} and \cref{lem:telescope}, we see that \[
       \norml{x^{\alpha+1}f}{1}{[0,1]} \lesssim
           \sum_{l=0}^{\infty}(l+1)(c_{l}+c_{l+2}-2c_{l+1}) = c_0 < \infty, 
    \] completing the proof of our claim.

    We can now mimic the argument in \cite{GrafakosCFA}
    to show that \[
        f_m = j_m^\alpha c_m \geq a_m \qand \frac{f_m}{j_m^\alpha} = c_m \to 0.
    \]

    A slight modification of that argument shows that we can find such an $f$
    with the additional property that $f_{2m+1}=0$ for all $m\in\N$.
    It suffices to replace the kernel $K^\alpha_M$ with the `even kernel' \[
        K^{\alpha,e}_M(x) := \sum_{m=1}^{M}\left(1-\frac{m}{M+1}\right)
            j_m^{\alpha}J_{0}(j_{2m}x).
    \] As can be seen from the proof of \cref{lem:xFbounds} below,
    we still have that $x^{\alpha+1}K^{\alpha,e}_M(x) \geq -C$
    uniformly in $M\in\N$ and $x\in[0,1]$.
    We construct $f$ as in \cref{eqn:fdefn} and show that it satisfies $x^{\alpha+1}f(x)\in L^1$,
    just as before.
    Then, the same computation of the $f_m$ shows that
    \begin{equation}\label{eqn:f2ms}
        f_{2m} = j_m^\alpha c_m \qand f_{2m+1} = 0 \qfa m \in \N.
    \end{equation}
    Indeed,
    \begin{align*}
         (K^{\alpha,e}_M)_m &= \sum_{k=1}^{M}\left(1-\frac{k}{M+1}\right)
            \frac{2j_m^{\alpha}}{J_1(j_m)^2}\int_0^1 J_0(j_m x)J_0(j_{2k}x)x\dm{x}\\
            &= \sum_{k=1}^{M}\left(1-\frac{k}{M+1}\right)j_m^{\alpha}\delta_{m,2k},
    \end{align*}
    which vanishes whenever $m>2M$ or $m$ is odd. Thus, \[
        (K^{\alpha,e}_M)_m =
        \begin{cases}
            \left(1-\frac{m/2}{M+1}\right)j_{m/2}^{\alpha}, &\text{if $m\leq 2k$ is even};\\
             0, &\text{if $m$ is odd or $m > 2k$}.
        \end{cases}
    \] \cref{eqn:f2ms} now follows.

    By making the obvious adjustments, one can also construct an $f$ satisfying \[
        f_{2m} = 0 \qand f_{2m+1} = j_m^{\alpha}c_m \qfa m\in\N.\qedhere
    \]
\end{proof}
\section{`Strong' decay theorems}\label{sec:sdthm}

Numerical simulations suggest that $K^\alpha_M(x)\geq 0$ uniformly in $x\in[0,1]$
and $M\in\N$, for the range $0\leq \alpha \leq \frac{1}{2}$%
---see \cref{fig:FMplots}.

\begin{figure}[ht]
\captionsetup[subfigure]{font=scriptsize,labelfont=scriptsize}
\centering
\begin{subfigure}{0.3\textwidth}
    \includegraphics[width=\textwidth]{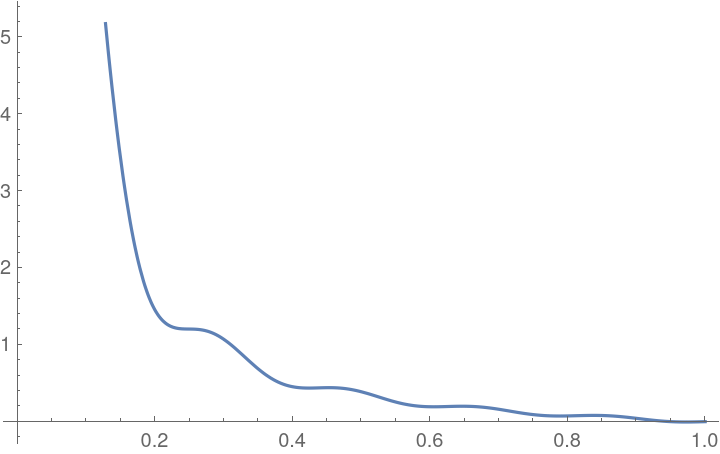}
    \caption{Plot of $K_{10}^{1/2}(x)$.}
    \label{fig:F10}
\end{subfigure}
\hfill
\begin{subfigure}{0.3\textwidth}
    \includegraphics[width=\textwidth]{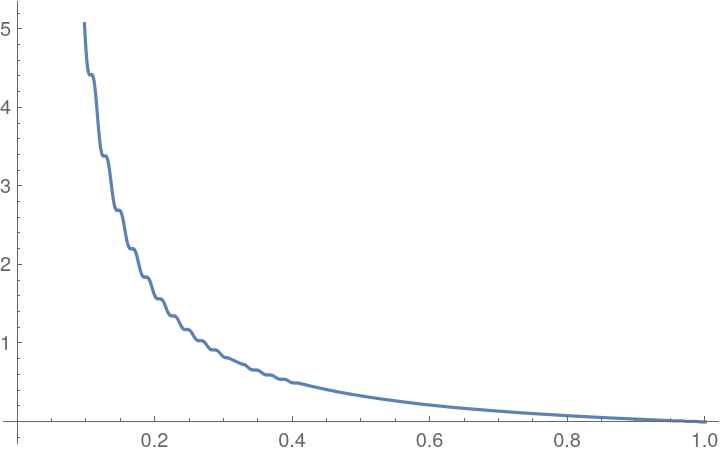}
    \caption{Plot of $K_{100}^{1/2}(x)$.}
    \label{fig:F100}
\end{subfigure}
\hfill
\begin{subfigure}{0.3\textwidth}
    \includegraphics[width=\textwidth]{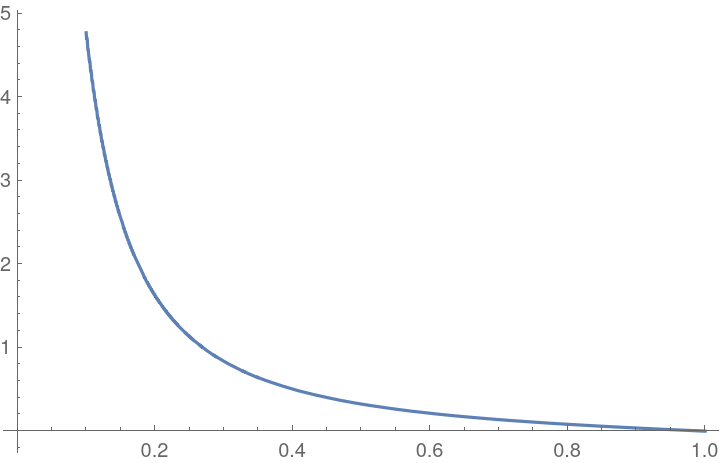}
    \caption{Plot of $K_{1000}^{1/2}(x)$.}
    \label{fig:F1000}
\end{subfigure}
\caption[Plots of $K^{1/2}_M(x)$]%
{Plots of $K^{1/2}_M(x)$ for $0\leq x \leq 1$ with $M=10$, $M=100$ and $M=1000$,
respectively. (Generated with Mathematica.)}
\label{fig:FMplots}
\end{figure}

We therefore make the following conjecture.
\begin{conj}\label{conj:KMpos}
    For all, $0\leq\alpha\leq\frac{1}{2}$, $M\in\N$ and $x\in [0,1]$, $K^{\alpha}_M(x)\geq 0$.
\end{conj}
Assuming this conjecture, we can prove the following strengthening of \cref{thm:mainthm}.
\begin{theorem}\label{thm:mainthmrootx}
    Assume that \cref{conj:KMpos} is true.
    Let $a_m\geq 0$ be a sequence such that $a_m\to 0$ and fix $0\leq \alpha < \frac{1}{2}$.
    Then there is a function $f$ such that $\sqrt{x}f(x)\in L^1$,
    whose Bessel--Fourier coefficients of order $0$, $f_m$, satisfy  \[
        0 \leq a_m \leq \frac{f_m}{j_m^\alpha} \to 0
            \qtext{as} m \to \infty.
    \] Furthermore, $f$ can be chosen so that $f_{2m}=0$ for all $m\in\N$,
    or so that $f_{2m+1} = 0$ for all $m\in\N$.
\end{theorem}
\begin{proof}
    Construct $f$ as before, using \cref{eqn:fdefn}.
    We now claim (assuming \cref{conj:KMpos}) that $\sqrt{x}f(x) \in L^1$.

    By assumption (i.e.~\cref{conj:KMpos}), $K^\alpha_M(x)\geq 0$, so \[
        \norml{\sqrt{x}K^\alpha_M}{1}{[0,1]} = \int_0^1 K^\alpha_M(x)\sqrt{x}\dm{x}.
    \]
    Writing $K^\alpha_M$ out as a sum and changing variables in the integral again,
    we have
    \begin{align}
        \int_0^1 K^\alpha_M(x)\sqrt{x}\dm{x}
            &= \sum_{m=1}^M\left(1-\frac{m}{M+1}\right)j_m^\alpha\int_0^1 J_0(j_m x)\sqrt{x}\dm{x}\nonumber\\
            &= \sum_{m=1}^M\left(1-\frac{m}{M+1}\right)\frac{j_m^\alpha}{j_m^{3/2}}
                    \int_0^{j_m}J_0(y)\sqrt{y}\dm{y}. \label{eqn:introotxKM}
    \end{align}
    We now split the range of integration into two pieces.
    When $0\leq y \leq 1$, we have $J_0(y)\sqrt{y} = O(1)$.
    When $1\leq y \leq j_m$, we can use the asymptotics for $J_0$ in \cref{lem:Jasym}.
    Hence,
    \begin{align}
        \frac{j_m^\alpha}{j_m^{3/2}}\int_0^{j_m}J_0(y)\sqrt{y}\dm{y} &= \frac{j_m^\alpha}{j_m^{3/2}}\int_0^1 J_0(y)\sqrt{y}\dm{y}\label{eqn:int01}\\
            &\quad+ \frac{j_m^\alpha}{j_m^{3/2}}\int_{1}^{j_m}\sqrt{\frac{2}{\pi}}\cos(y-\pi/4)\dm{y}\label{eqn:int1jmcos}\\
            &\quad+ \frac{j_m^\alpha}{j_m^{3/2}}\int_1^{j_m} O\left(\frac{1}{y}\right)\dm{y}.\label{eqn:int1jmo}
    \end{align} Now, using the asymptotics for $j_m$ in \cref{lem:zeros}, we have\[
        (\ref{eqn:int01}) \lesssim \frac{1}{m^{3/2-\alpha}}, \quad (\ref{eqn:int1jmcos}) \lesssim \frac{1}{m^{3/2-\alpha}}
        \qand (\ref{eqn:int1jmo}) \lesssim \frac{\log{m}}{m^{3/2-\alpha}} \lesssim \frac{1}{m^{3/2-\alpha/2}},
    \] since $m^{-\alpha/2}\log{m}\to 0$.
     For $\alpha<\frac{1}{2}$, all of the above terms are summable.
     Combining and summing these terms in \cref{eqn:introotxKM}, we deduce that the norms
     $\norml{\sqrt{x}K^\alpha_M}{1}{[0,1]}$ are bounded uniformly in $M$.
     Therefore, from the definition of $f$ in (\ref{eqn:fdefn})
     and \cref{lem:telescope} we see that \[
        \norml{\sqrt{x}f}{1}{[0,1]} \lesssim
            \sum_{l=0}^{\infty}(l+1)(c_{l}+c_{l+2}-2c_{l+1}) = c_0 < \infty, 
     \] completing the proof of our claim.

     The computation of the $f_m$, including the modifications
     to ensure that $f_{2m+1}=0$ for all $m\in\N$,
     are identical to those in the proof of \cref{thm:mainthm}.
\end{proof}
\begin{remark}
    We have reason to suspect that, if $\alpha=\frac{1}{2}$,
    then we can prove a similar result, with $xf(x)\in L^1$.
     
     In this case, the terms obtained from the decomposition in \cref{eqn:int01,eqn:int1jmo,eqn:int1jmcos}
     are no longer summable.
     However, if we continue to assume that $K_M^{1/2}(x)\geq 0$, then we can make use of \cref{eqn:intJ0rdr} to compute
      \begin{align*}
        \int_0^1 K^{1/2}_M(x)x\dm{x}
            &= \sum_{m=1}^M\left(1-\frac{m}{M+1}\right)\frac{\sqrt{j_m}}{j_m^2}
                    \int_0^{j_m}J_0(y)y\dm{y}\nonumber\\
            &= \sum_{m=1}^M\left(1-\frac{m}{M+1}\right)\frac{\sqrt{j_m}}{j_m^{2}}j_mJ_1(j_m)\\
            &= \sum_{m=1}^M\left(1-\frac{m}{M+1}\right)\frac{1}{\sqrt{j_m}}J_1(j_m).
    \end{align*}
    Numerical computations suggests that the signs of $J_1(j_m)$
    alternate, with $J_1(j_{2m}) < 0$ and $J_1(j_{2m+1}) > 0$.
    On the other hand, we expect that the values of $\abs{J_0'(j_m)}=\abs{J_1(j_m)}$
    form a decreasing sequence as $m\to\infty$ which, by \cref{lem:zeros},
    decays like $O\left(\frac{1}{\sqrt{j_m}}\right)$---see \cref{fig:J1decreasing}.
    Thus,
    \begin{align*}
        \int_0^1 K^{1/2}_M(x)x\dm{x}
            &\approx \sum_{m=1}^M\left(1-\frac{m}{M+1}\right)\frac{(-1)^{m+1}}{j_m}\\
            &= \frac{1}{M+1}\sum_{m=1}^{M}\sum_{n=1}^{m}\frac{(-1)^{m+1}}{j_m},
    \end{align*} which should remain bounded as $M\to\infty$ by the Alternating Series Test.
    It is then clear that, defining $f$ as in \cref{eqn:fdefn} with $\alpha=\frac{1}{2}$
    would yield $xf(x)\in L^1$, and the computation of the coefficients $f_m$ would goe through just as before.

    However, we have not been able to prove that the sequence $\left(\abs{J_1(j_m)}\right)_{m}$
    is decreasing, although it is strongly suggested by \cref{fig:J1decreasing}.

    \begin{figure}
        \centering
        \includegraphics[scale=.3]{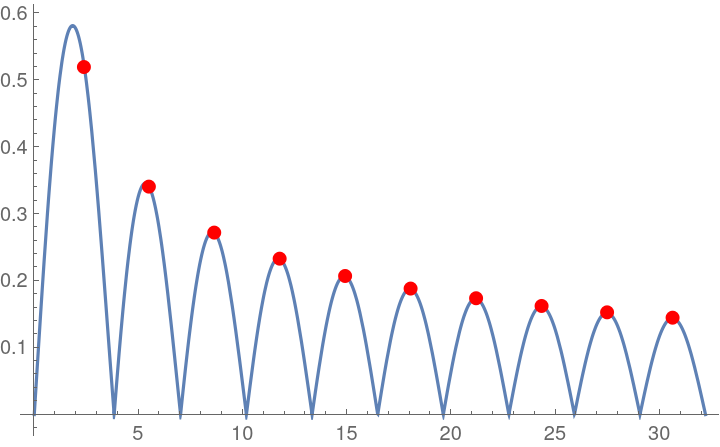}
        \caption{The values of $\abs{J_1(j_m)}$
        appear monotonically decreasing.}
        \label{fig:J1decreasing}
    \end{figure}
\end{remark}
\section{Bounds on $x^{\alpha+1}K^\alpha_M(x)$}\label{sec:KMbounds}

In this section, we prove \cref{lem:xFbounds}.
First we will need another technical lemma.
\begin{lemma}\label{lem:malphaebounds}
    For every $\alpha>0$ and $u>0$ \[
        \max_{m>0} \left\{m^\alpha\e^{-mu}\right\} =  \left(\frac{\alpha}{\e u}\right)^{\alpha} \lesssim u^{-\alpha}.
    \]
\end{lemma}
\begin{proof}
    Let $h(y) := y^\alpha \e^{-yu}$ and differentiate twice to find \[
        h'(y) =  \alpha y^{\alpha-1}\e^{-yu} + - uy^{\alpha}\e^{-yu},
    \] and \[
        h''(y) = \left[\alpha(\alpha-1)y^{\alpha-2}-2\alpha uy^{\alpha-1}+u^2y^\alpha\right]\e^{-uy}.
    \] It is then easy to see that the only solution to $h'(y)=0$ is $y=\frac{\alpha}{u}$
    and that \[
        h''\left(\frac{\alpha}{u}\right) = -\frac{\alpha^{\alpha-1}}{u^{\alpha-2}}\e^{-\alpha} < 0,
    \] so that $\max_{y>0}h(y) = h(\alpha/u) =\left(\frac{\alpha}{\e u}\right)^{\alpha}$.
\end{proof}

On to the proof of \cref{lem:xFbounds}.

To begin, note that it is sufficient to prove that $x^{\alpha+1}S^{\alpha}_M(x) \geq -C$ uniformly in $x$ and $M$.
    Using the asymptotics from \cref{eqn:J0asym}, we have that
    \begin{align}
        x^{\alpha+1}S^{\alpha}_M(x) &= x^{\alpha+1}\sum_{m=1}^{M}j_m^{\alpha}J_0(j_m x)\nonumber\\
                &= \sqrt{\frac{2}{\pi}}x^{\alpha+1/2}\sum_{m=1}^M\frac{\cos(j_m x-\pi/4)}{j_m^{1/2-\alpha}}
                 + \sum_{m=1}^{M}x^{\alpha+1}j_m^{\alpha}R(j_m x)\nonumber\\
                &\equiv \sqrt{\frac{2}{\pi}}x^{\alpha+1/2}\tilde{S}^{\alpha}_M(x) + R^{\alpha}_M(x)\label{eqn:xSMasymp},
    \end{align} where we have defined \[
        \tilde{S}^{\alpha}_M := \sum_{m=1}^M \frac{\cos(j_m x -\pi/4)}{j_m^{1/2-\alpha}} \qand
        R^{\alpha}_M(x) := \sum_{m=1}^{M}x^{\alpha+1}j_{m}^{\alpha}R(j_m x).
    \]
    It suffices to obtain lower bounds on $x^{\alpha+1/2}\tilde{S}^{\alpha}_M(x)$
    and upper bounds on $\abs{R^{\alpha}_M(x)}$
    that are uniform in $x\in[0,1]$ and $M\in\N$.

    \paragraph{Lower bounds on $\tilde{S}^{\alpha}_M$.} Since $x^{\alpha+1/2}\geq 0$,
    it suffices to obtain lower bounds on the sums $\tilde{S}^{\alpha}_M$.
    To this end, we follow \cite{Zygmund2003}
    and relate the sum to an integral. We claim that
    \begin{equation}\label{eqn:tildeS}
        \tilde{S}^{\alpha}_M(x) = \frac{x/2}{\sin(\pi x/2)}\int_0^{j_M+\frac{\pi}{2}}\frac{\cos(ux-\pi/4)}{u^{1/2-\alpha}}\dm{u}
            + O(1).
    \end{equation}
    To show this, note that, for each $m\in\N$, we have
    \begin{align*}
        \int_{j_m-\frac{\pi}{2}}^{j_m+\frac{\pi}{2}} \cos(ux-\pi/4)\dm{u}
            &= \left[\frac{\sin(ux-\pi/4)}{x}\right]_{j_m-\frac{\pi}{2}}^{j_m+\frac{\pi}{2}}\\
            &= \frac{2\sin(\pi x/2)}{x}\cos(j_m x -\pi/4),
    \end{align*}
    using the sum-to-product formula for a difference of sines \[
        \sin(a) - \sin(b) = 2\sin\left(\frac{a-b}{2}\right)\cos\left(\frac{a+b}{2}\right).
    \] Hence, integrating over all intervals $[j_m-\frac{\pi}{2},j_m+\frac{\pi}{2}]$
    for $1\leq m \leq M$, yields
    \begin{align*}
        \sum_{m=1}^M\int_{j_m-\frac{\pi}{2}}^{j_m+\frac{\pi}{2}}\frac{\cos(ux-\pi/4)}{u^{1/2-\alpha}}\dm{u}
        &= \sum_{m=1}^M\int_{j_m-\frac{\pi}{2}}^{j_m+\frac{\pi}{2}}\cos(ux-\pi/4)\left(\frac{1}{u^{1/2-\alpha}}-\frac{1}{j_m^{1/2-\alpha}}\right)\dm{u}\\
        &\quad+\sum_{m=1}^M\int_{j_m-\frac{\pi}{2}}^{j_m+\frac{\pi}{2}}\frac{\cos(ux-\pi/4)}{j_m^{1/2-\alpha}}\dm{u}\\
        &\leq \sum_{m=1}^{M}
            \sup_{\abs{u-j_m}<\frac{\pi}{2}}\abs{\frac{1}{u^{1/2-\alpha}}-\frac{1}{j_m^{1/2-\alpha}}}\\
        &\quad+\frac{\sin(\pi x/2)}{x/2}\sum_{m=1}^{M}\frac{\cos(j_m x -\pi/4)}{j_m^{1/2-\alpha}}.
    \end{align*} By the mean value theorem, for $\abs{u-j_m}<\frac{\pi}{2}$, \[
        \frac{1}{u^{1/2-\alpha}} - \frac{1}{j_m^{1/2-\alpha}} = \left(\alpha-\frac{1}{2}\right)\frac{j_m-u}{y^{3/2-\alpha}} \qtext{with}
        j_m-\frac{\pi}{2} \leq y \leq j_m+\frac{\pi}{2},
    \] whence we obtain, using \cref{eqn:jmm}, that \[
        \sup_{\abs{u-j_m}<\frac{\pi}{2}}\abs{\frac{1}{u^{1/2-\alpha}}-\frac{1}{j_m^{1/2-\alpha}}}
            \lesssim \frac{1}{j_m^{3/2-\alpha}} \lesssim \frac{1}{m^{3/2-\alpha}}.
    \]
    Using this bound we have that
    \begin{align}
        \sum_{m=1}^M\int_{j_m-\frac{\pi}{2}}^{j_m+\frac{\pi}{2}}\frac{\cos(ux-\pi/4)}{u^{1/2-\alpha}}\dm{u}
        -\frac{\sin(\pi x/2)}{x/2}&\sum_{m=1}^{M}\frac{\cos(j_m x -\pi/4)}{j_m^{1/2-\alpha}}\nonumber\\
        &\lesssim \sum_{m=1}^M\frac{1}{m^{3/2-\alpha}} = O(1) \label{eqn:sumcosint}
    \end{align} since $\alpha<1/2$.

    The next step is to relate \[
        \sum_{m=1}^M\int_{j_m-\frac{\pi}{2}}^{j_m+\frac{\pi}{2}}
            \frac{\cos(ux-\pi/4)}{u^{1/2-\alpha}}\dm{u}
        \qtext{to}
        \int_{0}^{j_M+\frac{\pi}{2}}\frac{\cos(ux-\pi/4)}{u^{1/2-\alpha}}\dm{u}.
    \] For this we need to recall the asymptotics for $j_m$: \[
        j_m = \pi m -\frac{\pi}{4} + \epsilon_m,
    \] where (according to McMahon’s asymptotic expansions for large zeros) \[
        \epsilon_m = \frac{1}{8\pi m - 2\pi} + O\left(\frac{1}{m^3}\right)
        \qtext{as} m\to\infty.
    \] (See the formula on p.~505 of \citealt{Watson1995}.)
    
    Thus, if we sum over all the intervals $[j_m-\frac{\pi}{2},j_m+\frac{\pi}{2}]$
    for $1\leq m \leq M$, we will obtain the intergral over
    $[j_1-\frac{\pi}{2},j_M+\frac{\pi}{2}]$,
    plus a sum of the errors obtained from the overlaps that occur around
    the midpoints \[
        \frac{j_{m+1}-j_m}{2} = \frac{\pi}{2} +\frac{\epsilon_{m+1}-\epsilon_m}{2}
        = \frac{\pi}{2} + O\left(\frac{1}{m^2}\right).
    \] These errors are therefore of order \[
        \int_{\frac{\pi}{2}-\frac{1}{m^2}}^{\frac{\pi}{2}+\frac{1}{m^2}}
            \frac{1}{u^{1/2-\alpha}}\dm{u} \lesssim \frac{1}{m^2},
    \] and so are summable in $m$. It follows that \[
        \sum_{m=1}^M\int_{j_m-\frac{\pi}{2}}^{j_m+\frac{\pi}{2}}
            \frac{\cos(ux-\pi/4)}{u^{1/2-\alpha}}\dm{u}
        = \int_{0}^{j_M+\frac{\pi}{2}}\frac{\cos(ux-\pi/4)}{u^{1/2-\alpha}}\dm{u} + O(1).
    \] As $(x/2)/\sin(\pi x/2) > 0$ is bounded above and below on $[0,1]$,
    \cref{eqn:tildeS} follows from this display and \cref{eqn:sumcosint}.

    Since we are interested in the behaviour of this integral as $M\to\infty$,
    we will replace $j_M+\frac{\pi}{2}$ by $M$. We now claim that
    \begin{equation}\label{eqn:cosint}
        \int_{0}^{M}\frac{\cos(ux-\pi/4)}{u^{1/2-\alpha}}\dm{u} \geq -\frac{C_1}{x^{\alpha+1/2}} \qtext{for some} C_1 > 0.
    \end{equation} After a change of variables and by using the identity \[
        \cos(a-\pi/4) = \frac{1}{\sqrt{2}}\cos(a) + \frac{1}{\sqrt{2}}\sin(a),
    \] we obtain \[
        \int_{0}^{M}\frac{\cos(ux-\pi/4)}{u^{1/2-\alpha}}\dm{u} =
        \frac{1}{x^{\alpha+1/2}\sqrt{2}}\left(\int_0^{Mx}\frac{\cos(v)}{v^{1/2-\alpha}}\dm{v}
        + \int_0^{Mx}\frac{\sin(v)}{v^{1/2-\alpha}}\dm{v}\right).
    \] Thus, it remains to bound the indefinite integrals \[
        G(y) := \int_0^{y}\frac{\cos(v)}{v^{1/2-\alpha}}\dm{v} \qand
        H(y) := \int_0^{y}\frac{\sin(v)}{v^{1/2-\alpha}}\dm{v}.
    \] We deal with the cosine integral, the sine case being easier.
    Since \[
        G'(y) = \frac{\cos(y)}{y^{1/2-\alpha}} \qand
        G''(y) = -\frac{\sin(y)}{y^{1/2-\alpha}}-\left(\\\alpha-\frac{1}{2}\right)\frac{\cos(y)}{y^{\alpha+1/2}},
    \] and $\alpha<1/2$, we find that the local minima of $G$ of occur precisely for
    $y_k =\frac{3\pi}{2}+2\pi k$ (for any $k\in\N_0$).
    These minima occur between consecutive zeros of the cosine.
    Since $v^{\alpha-1/2}$ is decreasing in $v$, the net contribution
    to the integral over this period is positive,
    since the negative part of the $\cos(v)$ is weighted against a smaller value
    of $v^{\alpha-1/2}$---see \cref{fig:cosrootint}.
    \begin{figure}[ht]
        \centering
        \includegraphics[scale=.25]{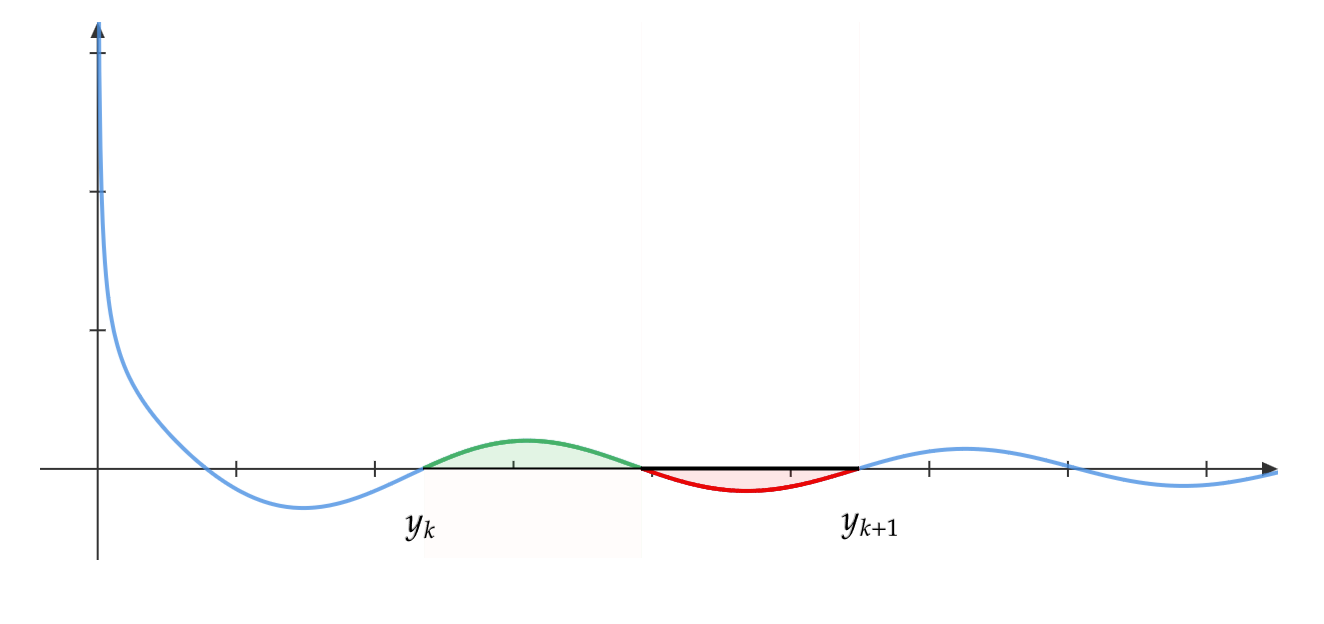}
        \caption[The function $\frac{\cos{v}}{v^{1/2-\alpha}}$]%
        {The function $\frac{\cos{v}}{v^{1/2-\alpha}}$.
        The positive green area outweighs the negative
        red area for a positive net contribution between
        $y_k$ and $y_{k+1}$.}
        \label{fig:cosrootint}
    \end{figure}
    More formally, \begin{align*}
        G(y_{k+1}) - G(y_k) &= 
            \int_{\frac{3\pi}{2}+2\pi k}^{\frac{3\pi}{2}+2\pi k +2\pi}
                \frac{\cos(v)}{v^{1/2-\alpha}}\dm{v}\\
            &= \int_{0}^{2\pi}\frac{\cos(w+3\pi/2+2\pi k)}{(w+3\pi/2+2\pi k)^{1/2-\alpha}}\dm{w}
                \quad(w = v - 3\pi/2-2\pi k)\\
            &= \int_{0}^{2\pi}\frac{\sin(w)}{(w+3\pi/2+2\pi k)^{1/2-\alpha}}\dm{w}
                \quad(\cos(a+3\pi/2)=\sin(a))\\
            &= \int_{0}^{\pi}\frac{\sin(w)}{(w+3\pi/2+2\pi k)^{1/2-\alpha}}\dm{w}\\
                &\qquad+\int_{\pi}^{2\pi}\frac{\sin(w)}{(w+3\pi/2+2\pi k)^{1/2-\alpha}}\dm{w}\\
     G(y_{k+1}) - G(y_k) &\geq \int_{0}^{\pi}\frac{\sin(w)}{(\pi+3\pi/2+2\pi k)^{1/2-\alpha}}\dm{w}\\
                &\qquad+ \int_{\pi}^{2\pi}\frac{\sin(w)}{(\pi+3\pi/2+2\pi k)^{1/2-\alpha}}\dm{w}\\
            &=\frac{1}{(\pi+3\pi/2+2\pi k)^{1/2-\alpha}}\left(
                \int_{0}^{\pi}\sin(w)\dm{w}+\int_{\pi}^{2\pi}\sin(w)\dm{w}\right)\\
            &= 0.
    \end{align*} This shows that the local minima of $G$ are increasing.
    Thus, $G(Mx) \geq G\left(\frac{3\pi}{2}\right)$ uniformly in $M\in\N$ and $x\in[0,1]$,
    and, similarly, $H(Mx)\geq H\left(0\right) =0 $, proving \cref{eqn:cosint}.

    Combining \cref{eqn:tildeS,eqn:cosint} yields
    \begin{equation}\label{eqn:xStildbound}
        x^{\alpha+1/2}\tilde{S}^{\alpha}_M(x) \geq -C_1 - x^{\alpha+1/2}C_0 \geq -C/2
    \end{equation} uniformly in $M\in\N$ and $x\in[0,1]$.
    
    \paragraph{Bounds on $R^{\alpha}_M$.} Recall that $R^{\alpha}_M(x) := \sum_{m=1}^Mx^{\alpha+1}j_m^{\alpha}R(j_m x)$.
    By \cref{eqn:jmm}, $\e^{-j_m xt} \leq \e^{-mxt}$
    for all $x\in[0,1]$ and $t\in(0,\infty)$.
    On the other hand, $j_m \lesssim m$ for all $m\geq 1$.
    Taking $u=xt/2$ in \cref{lem:malphaebounds}, we have \[
        j_m^{\alpha}\e^{-mxt} \lesssim m^{\alpha}\e^{-mxt} = \frac{m^\alpha}{\e^{\frac{mxt}{2}}}\e^{-\frac{mxt}{2}} \lesssim \frac{\e^{-\frac{mxt}{2}}}{(xt)^\alpha}.
    \]
    Hence, from \cref{eqn:J0err} we obtain the following bounds for all $x>0$:
    \begin{align*}
        \abs{R^{\alpha}_M(x)} 
            &\leq x^{\alpha+1}\int_0^{2}\sum_{m=1}^{\infty}m^{\alpha}\e^{-mxt}\sqrt{t}\dm{t}
                + x^{\alpha+1}\int_{2}^{\infty}\sum_{m=1}^{\infty}\frac{m^{\alpha}\e^{-mxt}}{t}\dm{t}\\
            &\lesssim \int_0^{2}\sum_{m=1}^{\infty}\e^{-\frac{mxt}{2}}\frac{x^{\alpha+1}\sqrt{t}}{(xt)^{\alpha}}\dm{t}
                + \int_{2}^{\infty}\sum_{m=1}^{\infty}\e^{-\frac{mxt}{2}}\frac{x^{\alpha+1}}{x^\alpha t^{\alpha+1}}\dm{t}\\
            &= \int_0^{2}\frac{\e^{-xt/2}}{1-\e^{-xt/2}}\frac{x\sqrt{t}}{t^{\alpha}}\dm{t}
                + \int_{2}^{\infty}\frac{\e^{-xt/2}}{1-\e^{-xt/2}}\frac{x}{t^{\alpha+1}}\dm{t}\\
            &= \int_0^{2}\frac{1}{\e^{xt/2}-1}\frac{x\sqrt{t}}{t^{\alpha}}\dm{t}
                + \int_{2}^{\infty}\frac{1}{\e^{xt/2}-1}\frac{x}{t^{\alpha+1}}\dm{t}\\
            &\leq \int_0^{2}\frac{2}{xt}\cdot\frac{x\sqrt{t}}{t^{\alpha}}\dm{t}
                + \int_{2}^{\infty}\frac{2}{xt}\cdot\frac{x}{t^{\alpha+1}}\dm{t}\\
            &\lesssim \int_0^2 \frac{1}{t^{\alpha+1/2}}\dm{t} + \int_2^{\infty} \frac{1}{t^{\alpha+2}}\dm{t}
            \leq C/2.
    \end{align*} uniformly in $x>0$ and $M\in\N$.
    Hence, we can take $x\to 0^+$ and obtain uniform bounds on $\abs{R^{\alpha}_M(x)}$
    for $x\in[0,1]$ and $M\in\N$.
    In conjunction with \cref{eqn:xStildbound}, we deduce the result from \cref{eqn:xSMasymp}.

\begin{remark}
    A slight adaptation of the above argument also shows that the even kernel $K^{\alpha,e}_M(x)$
    is uniformly bounded below.
    In this case we use the integrals
    \[
        \int_{j_{2m}-\pi}^{j_{2m}+\pi} \cos(ux-\pi/4)\dm{u}
            = \frac{2\sin(\pi x)}{x}\cos(j_{2m} x -\pi/4)
    \]
    to obtain the relation \[
        \sum_{m=1}^{M}\frac{\cos(j_{2m}x-\pi/4)}{j_m^{1/2-\alpha}}
            = \frac{x/2}{\sin(\pi x)}\sum_{m=1}^{M}
                \int_{j_{2m}-\pi}^{j_{2m}+\pi}\frac{\cos(ux-\pi/4)}{u^{1/2-\alpha}}\dm{u}
            + O(1).
    \] This time the intervals $[j_{2m}-\pi,j_{2m}+\pi]$ almost partition the larger interval
    $[j_2-\pi, j_{2M}+\pi]$, with errors occurring around the midpoints \[
        \frac{j_{2m+2}-j_{2m}}{2} = \pi + O\left(\frac{1}{m^2}\right),
    \] and so we again have \[
        \sum_{m=1}^{M}\frac{\cos(j_{2m}x-\pi/4)}{j_m^{1/2-\alpha}}
            = \frac{x/2}{\sin(\pi x)}\int_{0}^{j_{2M}+\pi}\frac{\cos(ux-\pi/4)}{u^{1/2-\alpha}}\dm{u}
            + O(1).
    \]

    The remainder in this case takes on the form \[
        \sum_{m=1}^{M} x^{\alpha+1}j_m^{\alpha}R(j_{2m}x),
    \] and can be dealt with in the same way, writing\[
        j_m^{\alpha}\e^{-2mxt} \lesssim m^{\alpha}\e^{-2mxt}= \frac{m^\alpha}{\e^{mxt}}\e^{-mxt}
        \lesssim \frac{\e^{-mxt}}{(xt)^\alpha}.
    \]
\end{remark}
\section{Comments for further research}\label{sec:conc}

We mentioned in the introduction that our \cref{thm:mainthm}
does not quite give us that \cref{thm:RLforBF} is `sharp'.
Our result requires an extra $\sqrt{x}$, loses a factor $\sqrt{j_m}$ in the decay, and leaves the end-point $\alpha=\frac{1}{2}$ open.
We saw in the proof that these `limitations' essentially come from \cref{lem:xFbounds}, since we require higher
powers of $x$ to ensure that the remainder terms in the asymptotics for $J_0$
remain uniformly bounded as $x\to 0^{+}$.

If, on the other hand, we were able to show that $K_M^{\alpha}\geq 0$ directly,
then we obtain the stronger \cref{thm:mainthmrootx}, where we always get $\sqrt{x}f(x)\in L^1$
for any value of $0\leq \alpha < \frac{1}{2}$, and, assuming 
further that $\abs{J_1(j_m)}$ is decreasing,
we can show the end-point case with $xf(x)\in L^1$.

Based on our numerical calculations---see \cref{fig:FMplots}---%
we make the following conjecture.
\begin{conj}\label{conj:decay}
    Given any non-negative sequence $a_m \to 0$,
    there is an $L^1(x\dm{x})$ function $f$
    whose Bessel--Fourier coefficients of order $0$, $f_m$, satisfy \[
        0 \leq a_m \leq \frac{f_m}{\sqrt{j_m}} \to 0 \qtext{as} m \to \infty.
    \]
\end{conj}

In order to prove the existence of an $L^1$ function on the disc,
whose 2-dimensional Bessel--Fourier series is not summable,
we need to make an additional assumption.
\begin{conj}\label{conj:evenodd}
    The function in \cref{conj:decay} can be chosen such that \[
        f_{2m} = 0 \qand f_{2m+1} \geq \sqrt{j_m}a_m \qfa m\in\N,
    \] or with \[
        f_{2m} \geq \sqrt{j_m}a_m \qand f_{2m+1} = 0 \qfa m\in\N.
    \]
\end{conj}
This conjecture, while a natural generalisation of \cref{thm:RLforBF} to $\alpha=\frac{1}{2}$, is less certain.
Indeed, when $\alpha<\frac{1}{2}$, we were able to `eliminate' the even (or odd) coefficients
because we had enough powers of $\frac{1}{j_m}$ to ensure that the kernel $K_M^\alpha$
was absolutely summable in $m$ as $M\to\infty$.
But when $\alpha=\frac{1}{2}$, we have to invoke oscillations in order to retain convergence,
and so both even and odd coefficients must make an appearance.
A search for an alternative proof that $K_M^\alpha\geq 0$ might yield a way around this `limitation'.

Assuming both \cref{conj:decay} and \cref{conj:evenodd}, we can prove the following proposition.
\begin{prop}\label{prop:L1div}
    Assume that \cref{conj:decay} and \cref{conj:evenodd} are both true.
    Then, there is a function $f\in L^1(\D)$ such that the partial sums \[
        S_{M,N}(r,\theta) = \sum_{n=-N}^{N}\sum_{m=1}^{M}c_{n,m}(f)
            J_{\abs{n}}(j_{\abs{n},m}r)\e^{\tpi n\theta}
    \] diverge in $L^1(r\dm{r}\dm{\theta})$ for any choice of $M=M_k, N=N_k\to\infty$
    as $k\to\infty$.
\end{prop}
\begin{proof}
    If the conjectures are true, then by taking $a_m = \frac{1}{\log{m}}$,
    we can find a function $g\in L^1([0,1], x\dm{x})$
    whose 1-dimensional Bessel--Fourier coefficients satisfy and \[
        g_{2m} = 0 \qand g_{2m+1} \geq \frac{\sqrt{j_{0,m}}}{\log(m)}.
    \] Set $f(r,\theta) := g(r)$, so that the 2-dimensional Bessel--Fourier series of $f$
    is given by \[
        S_{M,N}f(r,\theta) = \sum_{m=1}^{M}g_m J_0(j_{0,m}r),
    \] and so \[
        \norml{S_{M,N}f}{1}{\D} \geq 
        \abs{\int_0^1 S_{M,N}f(r,\theta)r\dm{r}} =
        \abs{\sum_{m=1}^{M}g_m \int_0^1 J_0(j_{0,m}r)r\dm{r}}.
    \] Using \cref{eqn:intJ0rdr} we have \[
        \int_0^1S_{M,N}f(r,\theta)r\dm{r}
            = \sum_{m=1}^{M}\frac{g_m}{j_{0,m}^2}\int_{0}^{j_{0,m}}J_0(y)y\dm{y}
            = \sum_{m=1}^{M}\frac{g_m}{j_{0,m}}J_1(j_{0,m}).
    \] Now, since $J_1 = -J_0'$, it is not hard to see that the signs of $J_1(j_{0,m})$
    alternate, with $J_1(j_{0,2m}) < 0$ (i.e., $J_0$ is decreasing at $j_{0,2m}$)
    and $J_1(j_{0,2m+1}) > 0$ (i.e., $J_0$ is increasing at $j_{0,2m+1}$).
    Since $g_{2m}=0$ for all $m$, the partial sums are non-negative and therefore,
    using \[
        \frac{\abs{J_1(j_{0,m})}}{\sqrt{j_{0,m}}} \gtrsim \frac{1}{j_{0,m}}
    \]
    from \cref{eqn:jmJ1jmasym}, we obtain \[
        \norml{S_{N,M}f}{1}{\D} = \sum_{m=1}^{M}\frac{g_{2m+1}}{\sqrt{j_{0,2m+1}}}\frac{\abs{J_1(j_{0,2m+1})}}{\sqrt{j_{0,2m+1}}}
        \gtrsim \sum_{m=1}^{M}\frac{1}{m\log(2m+1)} \to \infty,
    \] as required.
\end{proof}

\section*{Acknowledgements}
I am grateful to my advisor, Prof.~James C.~Robinson, for helpful guidance and discussion
while working on this project.

I acknowledge funding and support from EPSRC and The University of Warwick during the 
preparation on this manuscript.

\bibliography{bibliography}

\end{document}